\documentclass[aip,pof,reprint]{revtex4-2}

\makeatletter
\frontmatter@init
\let\frontmatter@init\@empty
\makeatother

\usepackage[T1]{fontenc}
\usepackage{amsmath,amssymb,amsthm}
\usepackage{bm}
\usepackage{booktabs}
\usepackage{graphicx}
\usepackage{url}
\usepackage[expansion=false,protrusion=false]{microtype}

\theoremstyle{plain}
\newtheorem{proposition}{Proposition}[section]
\newtheorem{corollary}[proposition]{Corollary}
\theoremstyle{definition}
\newtheorem{definition}[proposition]{Definition}

\newcommand{\parder}[2]{\frac{\partial #1}{\partial #2}}

\title{Uniform spacetime dilation of the magnetohydrodynamic balance
system: covariant form, invariant dimensionless groups, and a
near-horizon illustration}

\author{Yuanya Li}
\affiliation{Department of River and Harbor Engineering,
Nanjing Hydraulic Research Institute, Nanjing, China}
\email{yyli@nhri.cn}

\begin{document}

\begin{abstract}
The ideal magnetohydrodynamic balance system---Maxwell's equations
together with the conservation laws of mass, momentum, and energy---is
form-invariant under a uniform four-dimensional dilation of flat
spacetime, $(t,\bm x)=(L t',L\bm x')$ with $L>1$, provided the current
density, the transport coefficients, and the external heating and
cooling rates are rescaled according to a fixed dictionary:
$\bm j=\bm j'/L$, $\sigma'=\sigma L$, $\eta'=\eta/L$,
$Q'-\Lambda'=L(Q-\Lambda)$. We derive the dictionary and identify the
dilation as an equivalence transformation of the resistive system,
not a symmetry: it carries equations and their coefficients into
equations of the same form. Informal use ignores that the invariance
of every dimensionless group---magnetic Reynolds and Lundquist
numbers, plasma beta, Mach numbers, all ratios of dynamical to
resistive timescales---belongs to the transformation realised as a
whole. A coordinate-only realisation, which is what a numerical zoom
performs, divides $\mathrm{Rm}$ and $S$ by $L$ and reorders the
timescales by a factor $L$; a hundredfold refinement inflates the
Sweet--Parkinson rate tenfold. ``The resistivity vanishes as
$L\to\infty$'' is therefore a statement about units. We separate three legitimate readings of
the map---patch localisation, clock-and-sampling bookkeeping, and
gradient rescaling---and state what each does not deliver: linearisation
is controlled by the perturbation amplitude, not by $L$; shortened
primed timescales are a playback convention, not a speed-up; and
$L$-fold steeper plotted gradients do not add resolving power. A
Rindler-wedge model isolates the one place where such a scale carries
physical content---the gravitational redshift---and yields the proper
distance inside which horizon-scale transients are stretched beyond
day-scale (M87*) and minute-scale (Sgr~A*) observing cadences.
\end{abstract}

\keywords{scale covariance; magnetohydrodynamics; similarity transformation;
magnetic Reynolds number; Rindler horizon; accretion flows}

\maketitle

\section{Introduction}
\label{sec:intro}

Magnetohydrodynamics (MHD) in astrophysics is not a separate theory but
the same local balance structure used everywhere a conducting fluid
carries a magnetic field: Maxwell's equations coupled to the conservation
of mass, momentum, and energy \cite{priest2014mhd}. What distinguishes
the solar wind, a laboratory reconnection experiment, and a hot accretion
flow around a black hole is the parameter regime---the magnetic Reynolds
number, the plasma beta, the ordering of dynamical and dissipative
timescales---and the background against which the balance is solved
\cite{ji2022reconnection,guo2024reconnection,yao2022vortex}. The
relativistic branch of that statement is the workhorse of black-hole
imaging and jet modelling \cite{mizuno2022grmhd,jiang2023loops,
dihingia2022temp,cattorini2021binary,ricarte2023spin,aktar2023torus},
with the electromagnetic extraction of a rotating hole's energy that
underlies the jet picture formulated early on
\cite{blandfordznajek1977}, and the observations it must confront are
now resolve-level
\cite{eht2019m87,eht2022sgra,eht2022sgraVI,eht2024m87,broderick2023shadow,
chael2021innershadow,johnson2023ngeht,pesce2023ngeht}.

Because these flows look geometrically self-similar---thin discs, collimated
jets, funnels that repeat their shape over decades in radius---practitioners
reach for scale transformations. Two of these are dynamical in the strict
sense and are widely used: self-similar solutions, in which one solution
maps onto another with the equations doing the work
\cite{narayan1995,blandfordpayne1982,narayan2012outflow}, and dynamic
similarity scaling, in which a magnetised implosion is extrapolated to
another driver by scaling length, density, and field together
\cite{ruiz2023similarity}. In both cases a physical statement is being
made: a specified solution of the balance laws is carried into another
solution of the same laws.

A third operation is purely coordinate-kinematic and is also used, more
informally. One adopts units in which lengths are divided by a factor
$L>1$---equivalently, one rescales the coordinates uniformly---and then
reads off consequences as if they were physical. Three such readings recur
in the literature and in numerical practice: that on the ``magnified'' patch
the flow linearises and may be treated as a wave system; that on the
rescaled clock dissipative and dynamical timescales shorten, so that
transients become slow-motion replays on the physical clock; and that
wavelengths shrink to $\lambda/L$ and gradients steepen by $L$, so that
fine structure is brought into view. A fourth step is sometimes appended,
in which $L\to\infty$ is taken and the resistivity is said to vanish.

The purpose of this paper is to do that bookkeeping completely, and then
to state what the bookkeeping cannot deliver. Our results are the
following. (i)~The closed balance system, including Ohmic dissipation
in the induction equation and in the internal-energy balance, is
exactly form-invariant under the uniform dilation, with a dictionary
fixed equation by equation (Sec.~\ref{sec:dictionary},
Table~\ref{tab:dictionary}). (ii)~The map is an equivalence
transformation, not a symmetry: because the coefficients are carried
with the equations, every dimensionless group---magnetic Reynolds and
Lundquist numbers, plasma beta, Mach numbers, ratios of dynamical to
resistive timescales---takes the same value before and after the map,
so the map stays inside one dimensionless class; whereas a
coordinate-only realisation, which is what a numerical zoom performs,
divides those two numbers by $L$ and shifts the timescale ordering by
$L$ (Proposition~\ref{prop:invariance}). Neither branch carries a flow
into a different regime, and the inferences usually drawn from the
dilation---that a flow becomes ideal, or better resolved, or slower---
therefore fail (Corollary~\ref{cor:forbidden}). A damped Alfv\'en wave
tests the dictionary against a known dispersion relation
(Sec.~\ref{sec:groups}). (iii)~The
three readings above survive as readings---each with a precise, and
restricted, content, and each with a stated negative
(Sec.~\ref{sec:readings}). (iv)~Genuine scale change in this setting
enters through the gravitational redshift of a near-horizon geometry,
not through the dilation; a Rindler-wedge model gives the quantitative
statement that the informal ``frozen horizon'' narrative is reaching for,
together with the proper distance inside which it holds
(Sec.~\ref{sec:rindler}); written in the source's dimensionless
variables that statement becomes a relation between the observed
duration and the emission radius of a transient, and it is stated in a
form that one observation can falsify.

We make no statement here about black-hole uniqueness or about the
multipole structure of a stationary horizon; the classical results are
cited, not re-derived \cite{chrusciel2012review}, and the extremal
near-horizon geometry lies outside our scope
\cite{bardeenhorowitz1999}. Nor do we construct a global nonlinear
solution: the content of the paper is the dictionary, its invariant
structure, and the resulting constraints on what may be inferred from
a rescaling. Lengths and fields are in Gaussian units throughout.

\section{The balance system and its dimensional structure}
\label{sec:system}

\subsection{Equations}

For a conducting fluid of density $\rho$, velocity $\bm v$, pressure $P$,
temperature $T$, specific entropy $S$, magnetic field $\bm B$, electric
field $\bm E$, and current density $\bm j$, the closed system in Gaussian
units reads
\begin{align}
&\parder{\rho}{t}+\nabla\cdot(\rho\bm v)=0,
\label{eq:mass}\\
&\rho\Bigl(\parder{\bm v}{t}+\bm v\cdot\nabla\bm v\Bigr)
 =-\nabla P+\frac{1}{c}\,\bm j\times\bm B,
\label{eq:mom}\\
&\parder{\bm B}{t}=\nabla\times(\bm v\times\bm B)
 +\frac{c^{2}}{4\pi\sigma}\nabla^{2}\bm B,
\label{eq:ind}\\
&\bm j=\frac{c}{4\pi}\nabla\times\bm B,
\label{eq:ampere}\\
&\bm j=\sigma\Bigl(\bm E+\frac{1}{c}\bm v\times\bm B\Bigr),
\label{eq:ohm}\\
&\rho T\Bigl(\parder{S}{t}+\bm v\cdot\nabla S\Bigr)
 =Q-\Lambda+\frac{\eta}{4\pi}\,|\nabla\times\bm B|^{2},
\label{eq:energy}
\end{align}
with magnetic diffusivity $\eta=c^{2}/(4\pi\sigma)$, external heating
$Q$, and external cooling $\Lambda$. Equation~\eqref{eq:ind} is written
in its resistive (non-ideal) form because that term is the one the
informal arguments of Sec.~\ref{sec:intro} invoke.

\subsection{Dimensional structure}

Two groups of quantities behave differently. The \emph{state} variables
$\{\rho,P,\bm v,\bm B,\bm E,T,S\}$ are fields whose values at a labelled
event are what an observer measures there; a change of units does not
alter those values. The \emph{coefficients and derived currents}
$\{\bm j,\sigma,\eta,Q,\Lambda\}$ carry the units of the equations they
appear in, and their numerical values do change when the units change.
The distinction is the whole origin of the dictionary below: the map
will leave the first group alone and rescale the second, and the
rescaling is not free---it is fixed, equation by equation, by the requirement
that the equations keep their form.

The dilation is not a new object in the theory of these equations.
Fluid and magnetohydrodynamic systems are invariant under a
Galilean-similitude algebra that contains a uniform dilation among its
generators, and the subalgebra classification, the invariant solutions
built on it, and the corresponding conservation laws have been worked
out \cite{galasrichter1991,webb2005lie,xia2022lie,sinuvasan2023}. That
literature settles one question at the outset: with the coefficients
held fixed, the dilation is a symmetry, and what it carries is
solutions into solutions. Resistivity, Ohmic heating, and external
heating and cooling change that status. The dilation can then survive
only if the coefficients are carried with it, and the question becomes
which constraints on them preserve the form of the equations. That
question is answered in Sec.~\ref{sec:dictionary}, and the answer,
collected in Table~\ref{tab:dictionary}, is what the rest of the paper
uses.

\section{The uniform dilation and the rescaling dictionary}
\label{sec:dictionary}

\subsection{The map}

On flat spacetime, let
\begin{equation}
\Phi_{L}:\ (t',x'^{i})\mapsto(t,x^{i})=(L t',\,L x'^{i}),\qquad L>1,
\label{eq:map}
\end{equation}
a smooth, orientation-preserving diffeomorphism with constant Jacobian
$\partial(t,x^{i})/\partial(t',x'^{i})=L^{4}$ and inverse
$x'^{i}=x^{i}/L$. Coordinate differentials satisfy
\begin{equation}
\mathrm{d}t=L\,\mathrm{d}t',\qquad
\parder{}{x^{i}}=\frac{1}{L}\parder{}{x'^{i}},\qquad
\nabla=\frac{1}{L}\nabla',
\label{eq:ops}
\end{equation}
so that a coordinate volume element transforms as
$\mathrm{d}^{3}x=L^{3}\mathrm{d}^{3}x'$. The map is not an isometry: it
stretches all four coordinates, and it therefore does not preserve
distances or any dimensionful invariant. What it preserves
is the \emph{form} of equations, in the sense made precise next.

We adopt the passive reading throughout. The field values are held fixed
as functions of the labelled event,
\begin{equation}
\rho'(x')=\rho(x),\quad P'(x')=P(x),\quad \bm v'(x')=\bm v(x),\quad
\bm B'(x')=\bm B(x),\quad \bm E'(x')=\bm E(x),
\label{eq:passive}
\end{equation}
and Eq.~\eqref{eq:map} is used only to rewrite the differential operators
and the coefficients. This is the reading that all three ``effects'' of
the informal literature rely on, so it is the reading we test.

\subsection{Equation-by-equation bookkeeping}

\emph{Mass.} Substituting Eq.~\eqref{eq:ops} into \eqref{eq:mass} and
multiplying by $L$ gives
$\partial_{t'}\rho'+\nabla'\cdot(\rho'\bm v')=0$, identically: mass
conservation is form-invariant with no rescaling at all.

\emph{Momentum.} The left side of \eqref{eq:mom} scales as $1/L$ for
each term. On the right, $\nabla P$ scales as $1/L$, while
$\bm j\times\bm B$ must scale as $1/L$ as well; since $\bm B$ is held
fixed, this requires $\bm j'=L\bm j$, i.e.
\begin{equation}
\bm j=\bm j'/L .
\label{eq:jscale}
\end{equation}

\emph{Amp\`ere's law.} With $\nabla'=L\nabla$, Eq.~\eqref{eq:ampere}
becomes $\bm j'=(c/4\pi)\nabla'\times\bm B'$, and consistency with
\eqref{eq:jscale} is automatic. No further input is needed.

\emph{Ohm's law.} Substituting \eqref{eq:jscale} into \eqref{eq:ohm},
and using that $\bm E$ and $\bm v\times\bm B$ are held fixed, gives
$\bm j'=\sigma L\,(\bm E'+\bm v'\times\bm B'/c)$, so the conductivity
rescales as
\begin{equation}
\sigma'=\sigma L .
\label{eq:sigmascale}
\end{equation}

\emph{Induction.} Equation~\eqref{eq:ind} scales as follows: the
left side and the convective term both as $1/L$, the diffusive term as
$1/L^{2}$. Multiplying by $L$ and using
$\eta=c^{2}/(4\pi\sigma)$,
\begin{equation}
\eta'=\eta/L ,
\label{eq:etascale}
\end{equation}
which is exactly the value implied by Eq.~\eqref{eq:sigmascale}. The
mutual consistency $\sigma'\eta'=\sigma\eta$ is a nontrivial check on
the dictionary: the two coefficients cannot rescale independently, and
any proposal that rescales one without the other breaks either
Amp\`ere's law or the induction equation.

\emph{Energy.} In \eqref{eq:energy} the left side scales as $1/L$;
multiplying the substituted equation by $L$, and using
$|\nabla'\times\bm B'|^{2}=L^{2}|\nabla\times\bm B|^{2}$ together with
\eqref{eq:etascale}, the Ohmic term becomes
$L\,(\eta/4\pi)|\nabla\times\bm B|^{2}$, so the heating and cooling
rates must carry the same factor:
\begin{equation}
Q'=LQ,\qquad \Lambda'=L\Lambda ,
\label{eq:heatscale}
\end{equation}
one particular choice of the pair, since only the difference enters
the balance.
The internal-energy balance is then form-invariant with
$S'=S$, $T'=T$, exactly as the mechanical equations are. The
magnetic-energy identity deserves the same check, since it is the one
place where the bookkeeping can be mistyped: written in primed
variables, every term of
$\partial_{t}u_{B}=-\nabla\cdot\bm S-\bm j\cdot\bm E$ acquires exactly
one factor $L$, the derivative terms through the chain rule and the
work term through $\bm j'=L\bm j$, so the two energy identities close
consistently rather than one against the other.

The outcome is collected in Table~\ref{tab:dictionary}.

\begin{proposition}[Scale covariance]
\label{prop:covariance}
The closed system \eqref{eq:mass}--\eqref{eq:energy} is form-invariant
under the uniform dilation \eqref{eq:map} if and only if the coefficient
and derived quantities transform as
$\bm j'=L\bm j$, $\sigma'=\sigma L$, $\eta'=\eta/L$,
$Q'-\Lambda'=L(Q-\Lambda)$, while $\rho$, $P$, $\bm v$, $\bm B$, $\bm E$,
$T$, and $S$ are held fixed. Equivalently:
\begin{equation}
\{\rho,P,\bm v,\bm B,\bm E,T,S\}\ \text{invariant},\qquad
\bm j,\ Q-\Lambda\ \text{scale as}\ L,\qquad
\sigma^{-1},\ \eta\ \text{scale as}\ L^{-1},
\end{equation}
in which the heating and cooling rates enter only through their
difference: $Q'=LQ+f$ and $\Lambda'=L\Lambda+f$ satisfy the same
condition for any function $f$ of the labelled event, so the pair is
fixed only up to that common addition.
\end{proposition}

\begin{table}[htbp]
\caption{Rescaling dictionary for the uniform dilation
\eqref{eq:map}. ``Invariant'' means the field is held fixed as a
function of the labelled event, Eq.~\eqref{eq:passive}; ``$L$'' means
the numerical value in the primed system is $L$ times the unprimed
value.}
\label{tab:dictionary}
\begin{ruledtabular}
\begin{tabular}{lllp{0.30\linewidth}}
Quantity & Symbol & Primed value & Origin of the rescaling \\
\hline
Time, length & $t,x^{i}$ & $t=L t'$, $x^{i}=L x'^{i}$ & definition, Eq.~\eqref{eq:map} \\
Gradient & $\nabla$ & $\nabla'=L\nabla$ & chain rule, Eq.~\eqref{eq:ops} \\
Density, pressure & $\rho$, $P$ & invariant & passive reading \\
Velocity & $\bm v$ & invariant & passive reading \\
Magnetic field & $\bm B$ & invariant & passive reading \\
Electric field & $\bm E$ & invariant & passive reading \\
Temperature, entropy & $T$, $S$ & invariant & passive reading \\
Current density & $\bm j$ & $\bm j'=L\bm j$ & momentum balance \\
Conductivity & $\sigma$ & $\sigma'=\sigma L$ & Ohm's law \\
Diffusivity & $\eta$ & $\eta'=\eta/L$ & induction equation \\
Heating, cooling & $Q-\Lambda$ & $Q'-\Lambda'=L(Q-\Lambda)$ & internal-energy balance; the pair is fixed only up to a common additive function \\
\end{tabular}
\end{ruledtabular}
\end{table}

Three remarks situate this result. First, the dictionary is forced by
form invariance, but not entry by entry: the induction equation fixes
$\eta'=\eta/L$ and Ohm's law fixes $\sigma'=\sigma L$, whose
consistency $\sigma'\eta'=\sigma\eta$ we have already used; the
momentum balance fixes $\bm j'=L\bm j$; the internal-energy balance
fixes only the difference $Q'-\Lambda'=L(Q-\Lambda)$. We therefore
state the heating and cooling entries as a difference, and make no
claim to a uniqueness that the equations do not deliver.

Second, the map is not a symmetry. A symmetry carries solutions into
solutions and leaves the dimensionful coefficients untouched, whereas
here $\eta$, $\sigma$, $Q$, and $\Lambda$ are themselves carried. What
the map carries is \emph{equations} into equations, together with the
coefficients that appear in them. Transformations of that kind are
called \emph{equivalence transformations} in the symmetry analysis of
systems with constitutive coefficients \cite{ovsiannikov1982}, and
Table~\ref{tab:dictionary} is their parameter table here. The
distinction is not terminological: it is what separates the dilation of
the ideal system, which is a symmetry of the Galilean-similitude
algebra, from the dilation of the resistive system, which is only an
equivalence transformation.

Third, because the coefficients carry units, the primed system is not a
different physical system; it is the same system described in units
rescaled by $L$. Everything in Sec.~\ref{sec:groups} turns on keeping
those two statements apart.

\section{What the equivalence transformation carries}
\label{sec:groups}

The dictionary of Table~\ref{tab:dictionary} looks like a licence to
change the regime of a flow: send $L\to\infty$, let $\eta'=\eta/L\to0$,
and declare the flow ideal. The result below forbids that, but not for
the reason a first reading suggests. It is not news that
$\mathrm{Rm}$ survives a change of units: no dimensionless number
ever fails to, and we said as much for $\beta$, $M_{A}$, and $M_{s}$
in Sec.~\ref{sec:dictionary}. What the result separates is the
transformation realised as a whole from its coordinate-only shadow,
and only the first of the two leaves the groups alone. That separation
is what an audit of the informal narrative needs, because a coordinate
zoom is precisely the partial object.

\begin{definition}
\label{def:groups}
For a flow with characteristic speed $v$, Alfv\'en speed
$v_{A}=B/\sqrt{4\pi\rho}$, sound speed $c_{s}$, and a structure of
proper size $\ell$, define
\begin{equation}
\mathrm{Rm}=\frac{v\ell}{\eta},\quad
S=\frac{v_{A}\ell}{\eta},\quad
\beta=\frac{8\pi P}{B^{2}},\quad
M_{A}=\frac{v}{v_{A}},\quad
M_{s}=\frac{v}{c_{s}},
\label{eq:groups}
\end{equation}
together with the three timescales
$t_{\mathrm{flow}}=\ell/v$, $t_{A}=\ell/v_{A}$,
$t_{\mathrm{diff}}=\ell^{2}/\eta$.
\end{definition}

\begin{proposition}[What the equivalence transformation carries]
\label{prop:invariance}
Let the dilation \eqref{eq:map} be realised with the dictionary of
Proposition~\ref{prop:covariance}. Then the structure of proper size
$\ell$ appears in the primed frame with size $\ell'=\ell/L$, and
\begin{equation}
\mathrm{Rm}'=\mathrm{Rm},\quad S'=S,\quad \beta'=\beta,\quad
M_{A}'=M_{A},\quad M_{s}'=M_{s},
\label{eq:invariance}
\end{equation}
while the three timescales shorten uniformly,
\begin{equation}
t'_{\mathrm{flow}}=t_{\mathrm{flow}}/L,\qquad
t'_{A}=t_{A}/L,\qquad
t'_{\mathrm{diff}}=t_{\mathrm{diff}}/L,
\label{eq:timescales}
\end{equation}
so that all of their ratios, and in particular
$t_{A}/t_{\mathrm{diff}}=1/S$, are invariant. If instead only the
coordinates are rescaled and the coefficients $\bm j$, $\sigma$,
$\eta$, $Q$, and $\Lambda$ keep their numerical values, the form of
the equations is destroyed and the groups move:
\begin{equation}
\mathrm{Rm}'\to\mathrm{Rm}/L,\qquad S'\to S/L,\qquad
\frac{t'_{A}}{t'_{\mathrm{diff}}}\to L\,\frac{t_{A}}{t_{\mathrm{diff}}}.
\label{eq:partial}
\end{equation}
\end{proposition}

\begin{proof}
By \eqref{eq:ops} a structure of proper size $\ell$ occupies the
primed interval $\ell'=\ell/L$. The velocities are invariant by
Eq.~\eqref{eq:passive}, and $\eta'=\eta/L$ by
Eq.~\eqref{eq:etascale}, so
$\mathrm{Rm}'=v'\ell'/\eta'=v(\ell/L)/(\eta/L)=v\ell/\eta=\mathrm{Rm}$,
and identically for $S$. The groups $\beta$, $M_{A}$, and $M_{s}$ contain
only state variables, which are invariant by construction. For the
timescales, $t'_{A}=\ell'/v'_{A}=t_{A}/L$ and
$t'_{\mathrm{diff}}=\ell'^{2}/\eta'=\ell^{2}/(L\eta)=t_{\mathrm{diff}}/L$;
each ratio is therefore unchanged. For the coordinate-only
realisation $\ell'=\ell/L$ while $\eta'=\eta$, so
$\mathrm{Rm}'=v(\ell/L)/\eta=\mathrm{Rm}/L$ and
$t'_{\mathrm{diff}}=\ell^{2}/(L^{2}\eta)=t_{\mathrm{diff}}/L^{2}$,
whence $t'_{A}/t'_{\mathrm{diff}}=L\,(t_{A}/t_{\mathrm{diff}})$.
\end{proof}

The two halves of the proposition answer different questions. The
coordinate-only half is not a curiosity of the algebra: it is the
transformation a practitioner performs when a grid is refined, and its
effect on the groups is a definite power of $L$ rather than a
vanishing one. The half that leaves the groups alone is the one in
which the coefficients are carried, and the reason they can be carried
is that the equations are form-invariant exactly when they are: the
requirement that the induction equation keep its form and the
requirement that $\mathrm{Rm}$ be preserved are one requirement, seen
once in the equation and once in the group. Written this way the
statement has content, because it tells an author which of the two
operations is being performed, and the answer is legible in
Eq.~\eqref{eq:partial}.

\begin{corollary}[What the dilation cannot be used to infer]
\label{cor:forbidden}
Under \eqref{eq:map}, none of the following can be concluded.\normalfont
\begin{enumerate}
\item \emph{That the flow becomes ideal as $L\to\infty$.} Under the
equivalence transformation $\eta'=\eta/L$ tends to zero, but so does
every length, and $\mathrm{Rm}$ is fixed. Under the coordinate-only
realisation $\mathrm{Rm}$ does move, but in the wrong direction,
$\mathrm{Rm}\to\mathrm{Rm}/L$, so a zoom makes a resistive flow look
more resistive, not less. Neither branch of
Eqs.~\eqref{eq:invariance} and~\eqref{eq:partial} supports ``the
resistivity vanishes''; a flow is ideal when $\mathrm{Rm}\gg1$ at
physical values of the coefficients.
\item \emph{That the flow changes regime.} $\beta$, $M_{A}$, and
$M_{s}$ are fixed, so the dilation cannot move a flow from
magnetically dominated to kinetically dominated, nor change the
ordering $t_{A}<t_{\mathrm{diff}}$ or its reverse.
\item \emph{That reconnection is faster or slower relative to the
Alfv\'en time.} The dimensionless reconnection rate is a function of
$S$ and of the geometry, and both $S$ and the geometric ratios are
invariant.
\item \emph{That the number of resolvable scales increases.} The
count of independent structures per unit proper length is
dimensionless and therefore invariant; Sec.~\ref{sec:readings}
makes this explicit for the gradient reading.
\end{enumerate}
\end{corollary}

A remark on the phrase ``$\eta\to0$'' closes the section. The diffusivity
has dimensions of length$^{2}$/time; a limit taken on it alone, with no
length and no time attached, is not a statement about any flow. The
dimensionless statement that the informal literature is reaching for is
$\mathrm{Rm}\to\infty$, and Proposition~\ref{prop:invariance} shows
that the dilation cannot deliver it. To make a flow more ideal one must
change the flow---raise $v\ell$, or lower $\eta$ in physical units
relative to the same $\ell$---not the units in which it is described.

\subsection{A check on a mode with a known dispersion relation}

An identity in algebra is worth exhibiting on a mode whose complex
frequency is known independently, because the exhibition tests the
dictionary as a whole rather than term by term. For a damped
Alfv\'en wave on a uniform background the dispersion relation is
\begin{equation}
\omega^{2}+\mathrm{i}\eta k^{2}\omega-k^{2}v_{A}^{2}=0,
\qquad
\omega=k v_{A}-\frac{\mathrm{i}}{2}\eta k^{2}
+O\!\left(\frac{\eta^{2}k^{3}}{v_{A}}\right),
\label{eq:alfven}
\end{equation}
in the weakly damped regime $\eta k^{2}\ll k v_{A}$. Under the map a
structure of proper wavelength $\lambda$ acquires the primed
wavenumber $k'=2\pi/\lambda'=Lk$, the Alfv\'en speed is invariant, and
$\eta'=\eta/L$, so the branch transforms as
\begin{equation}
\omega'=k'v'_{A}-\frac{\mathrm{i}}{2}\eta'k'^{2}
=L k v_{A}-\frac{\mathrm{i}}{2}\frac{\eta}{L}L^{2}k^{2}=L\,\omega ,
\label{eq:alfvencov}
\end{equation}
exactly and not merely to leading order. The quality factor
\begin{equation}
\mathcal{Q}\equiv\frac{\mathrm{Re}\,\omega}{|\mathrm{Im}\,\omega|}
=\frac{2v_{A}}{\eta k}
\label{eq:quality}
\end{equation}
is invariant, as it must be if the damping per cycle is a property of
the flow rather than of the units in which it is described. This is the
smallest complete demonstration of the dictionary: it exercises the
gradient entry, the diffusivity entry, and the time entry at once, and
it fails if any one of the three is rescaled alone.

\section{Three readings of the dictionary, and their limits}
\label{sec:readings}

The dilation does have content, provided one reads it as a bookkeeping
statement about a particular question. We isolate the three readings
that the literature uses, state the content of each in the form in
which it is true, and state its negative. Figure~\ref{fig:readings}
summarises them.

\subsection{Reading 1: patch localisation}

The useful statement is about the \emph{background}, not about the
dilation. Let a flow be described on a patch of proper size $\ell$ over
which the background (density, pressure, field) varies by a relative
amount $O(\ell/L_{\mathrm{bg}})$, and let the perturbation have
amplitude $\delta v$ against a characteristic speed $v_{\mathrm{char}}$.
The linearised system on that patch is
\begin{equation}
\begin{aligned}
&\partial_{t}\delta\rho+\nabla\cdot(\rho\,\delta\bm v)=0,\\
&\partial_{t}\delta\bm v+\bm v\cdot\nabla\delta\bm v
=-\frac{\nabla\delta P}{\rho}
+\frac{1}{4\pi\rho}(\nabla\times\delta\bm B)\times\bm B,\\
&\partial_{t}\delta\bm B
=\nabla\times(\delta\bm v\times\bm B)+\eta\,\nabla^{2}\delta\bm B,\\
&\delta P=c_{s}^{2}\,\delta\rho ,
\end{aligned}
\label{eq:lin}
\end{equation}
which closes $\delta\rho$, $\delta\bm v$, $\delta\bm B$, and
$\delta P$ on a uniform background and is the tangent-space system used
below; on a non-uniform background the last two terms carry an
additional $O(\ell/L_{\mathrm{bg}})$ correction that the second
condition of \eqref{eq:eps} controls. The system is valid when
\begin{equation}
\varepsilon\equiv\frac{\delta v}{v_{\mathrm{char}}}\ll1
\quad\text{and}\quad \frac{\ell}{L_{\mathrm{bg}}}\ll1 .
\label{eq:eps}
\end{equation}
Neither condition involves $L$. The dilation enters only by relabelling
which patch is called unit size: a patch that is small in the primed
frame is small in the unprimed frame as well, since $\ell'/\ell$ and
$L_{\mathrm{bg}}'/L_{\mathrm{bg}}$ are both $1/L$. \emph{The negative:}
the dilation does not linearise anything. If $\varepsilon$ is of order
unity, the flow is nonlinear in every frame, and no choice of $L$
changes that; the map rescales the coordinate extent of the patch
without rescaling the amplitude of the perturbation relative to the
background. Reading 1 is therefore a statement about scale separation
in the \emph{background}, and $L$ is a bookkeeping parameter, not a
control parameter.

\subsection{Reading 2: clock-and-sampling bookkeeping}

All three timescales shorten by the same factor $1/L$
[Eq.~\eqref{eq:timescales}], so a fixed window of primed coordinate time
$(t'_{1},t'_{2})$ contains $L$ times as many units as the corresponding
unprimed window. Written as a rule for records,
\begin{equation}
t=L\,t',\qquad N_{\mathrm{samples}}=L\,N'_{\mathrm{samples}}\Big|_{\text{fixed }\Delta t'},
\label{eq:clock}
\end{equation}
which is how a data-analysis pipeline that oversamples a signal behaves.
\emph{The negative:} nothing is sped up. Because the ratios in
Eq.~\eqref{eq:timescales} are invariant, the same number of wave periods,
Alfv\'en crossings, or resistive decay times fits into the same number of
characteristic times, before and after the map. A dissipative process
with $t_{\mathrm{diff}}/t_{A}=S$ in the unprimed description has the
same $S$ in the primed one; a transient that outlasts the observing
window outlasts it in both descriptions. ``Slow-motion replay'' describes
the relation between two clocks reading the same events, not a physical
slowing; the invariant statement is the ratio, and the ratio is fixed.

\subsection{Reading 3: gradient and wavelength rescaling}

A perturbation of proper wavelength $\lambda$ appears in the primed
frame at $\lambda'=\lambda/L$, and since $\nabla'=L\nabla$,
\begin{equation}
\lambda'=\lambda/L,\qquad
|\nabla'\delta\bm b'|=L\,|\nabla\delta\bm b| ,
\label{eq:grad}
\end{equation}
for any dimensionless field perturbation $\delta\bm b$. This is a
correct and useful identity: it tells one what a plotted profile looks
like after the units are changed. \emph{The negative:} it confers no
resolving power. The number of independent structures that fit inside a
fixed proper distance is a dimensionless count and is invariant
(Corollary~\ref{cor:forbidden}, item 4); a grid with $N$ points per
primed cell resolves the same number of physical modes as a grid with
$N$ points per unprimed cell, because the physical cell is larger by
$L$ and the mode is larger by $L$ in the same proportion. What changes
is the \emph{displayed} steepness, by exactly the factor in
\eqref{eq:grad}; what does not change is the information content. A
numerical scheme that refines its grid by $L$ while leaving
$\eta$ fixed does not implement the dilation either: it divides
$\mathrm{Rm}$ by $L$, and it therefore computes a different flow.

\begin{figure}[htbp]
\includegraphics[width=0.92\linewidth]{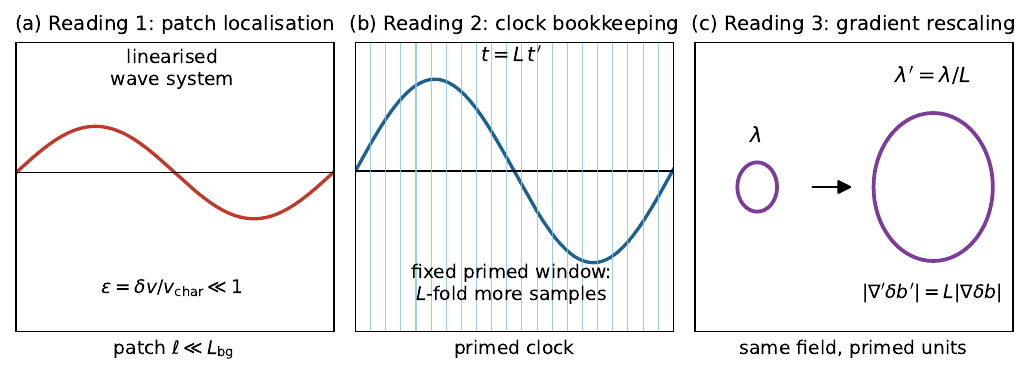}
\caption{The three readings of the dilation, drawn in the form in
which each is true. (a)~Reading 1: the linearised tangent-space system
on a patch whose proper size $\ell$ is small against the background
scale $L_{\mathrm{bg}}$; the controlling parameter is
$\varepsilon=\delta v/v_{\mathrm{char}}$, not $L$. (b)~Reading 2: on the
primed clock $t=L t'$, a fixed primed window holds $L$ times as many
samples; the ratios of timescales, and hence the physics, are unchanged.
(c)~Reading 3: the same field described in primed units has
$\lambda'=\lambda/L$ and $|\nabla'\delta\bm b'|=L|\nabla\delta\bm b|$;
plotted slopes steepen, resolving power does not.}
\label{fig:readings}
\end{figure}

\section{A near-horizon illustration: the Rindler wedge}
\label{sec:rindler}

\subsection{Where a scale does carry physical content}

The dilation of Sec.~\ref{sec:dictionary} is a coordinate operation on
flat spacetime, and Proposition~\ref{prop:invariance} shows that it
cannot move a flow between regimes. A near-horizon geometry is
different in one decisive respect: there the scaling of proper time
against coordinate time is a property of the metric, not of the choice
of units, and it does change what an exterior observer can resolve.

For a non-extremal Killing horizon the metric in the near-horizon
region, in coordinates adapted to a static exterior observer, is
\begin{equation}
\mathrm{d}s^{2}=-\alpha^{2}\mathrm{d}t^{2}
+\mathrm{d}\rho^{2}+r_{H}^{2}\,\mathrm{d}\Omega^{2},
\qquad \alpha=\kappa\rho ,
\label{eq:rindler}
\end{equation}
with $\rho$ the proper distance from the horizon, $\kappa$ the
surface gravity, and $\alpha=\sqrt{-g_{tt}}$ the redshift factor. For
a Schwarzschild hole $\kappa=1/(4M)$ in geometric units, so
$\alpha=\rho/(4M)$. The redshift does what no dilation can: it
multiplies the coordinate duration of a process by $1/\alpha$,
\begin{equation}
\tau_{\infty}=\frac{\tau_{\mathrm{prop}}}{\alpha}
=\frac{4M}{\rho}\,\tau_{\mathrm{prop}} ,
\label{eq:tauinf}
\end{equation}
indefinitely, as $\rho\to0$.

One caution about the comparison, since it is the reason the dilation
was introduced in the first place. Equation~\eqref{eq:rindler} is a
statement about a curved metric; Eq.~\eqref{eq:map} is a statement
about units on flat spacetime. The two coincide only in the
kinematic sense that both multiply coordinate durations by a
constant, and they differ in everything that matters for
Proposition~\ref{prop:invariance}: the redshift is a function of
position, it is not a diffeomorphism of the manifold onto itself in
the same sense, and it leaves $\mathrm{Rm}$, $\beta$, and $S$
untouched because it acts on time alone and not on lengths or on
$\eta$. The redshift therefore does not do the work that the
$L\to\infty$ step was asked to do in the informal narrative.

\subsection{When is a transient stretched beyond the cadence?}

Equation~\eqref{eq:tauinf} has a direct observational reading. Let
$\Delta t_{\mathrm{obs}}$ be the cadence with which a source is
sampled, and let a transient have proper duration
$\tau_{\mathrm{prop}}=\xi\,GM/c^{3}$ with $\xi$ of order unity, the
horizon light-crossing time. The condition that follows is one of
resolution and not of flux freezing: it asks whether the redshifted
transient outlasts the cadence, and it is distinct from the
$\mathrm{Rm}\gg1$ condition taken up below. A transient is
unresolvable at the cadence when
\begin{equation}
\tau_{\infty}=\frac{4M}{\rho}\,\tau_{\mathrm{prop}}
\ \gtrsim\ N\,\Delta t_{\mathrm{obs}},\qquad N\gg1 .
\label{eq:criterion}
\end{equation}
Table~\ref{tab:rindler} evaluates this for the two sources with
resolved horizon-scale structure, and
Fig.~\ref{fig:rindler} shows the same content graphically. The
threshold $N$ is a convention, and it is kept as a parameter rather
than fixed at ten: with
$\rho_{\max}=4\tau_{\mathrm{prop}}/(N\Delta t_{\mathrm{obs}})$, a
day-scale cadence for M87* \cite{eht2019m87,eht2024m87} gives
$\rho\lesssim1.5$, $0.49$, and $0.15\,GM/c^{2}$ for $N=1,3,10$,
and a ten-minute cadence for Sgr~A* \cite{eht2022sgra} gives
$\rho\lesssim0.13$, $0.044$, and $0.013\,GM/c^{2}$; the cadences
quoted are the nominal sampling intervals of those campaigns. The
windows are narrow, and they are set by the cadence rather than by the
geometry: improving the sampling by a factor $f$ moves the threshold
inward by the same factor, and a longer-lived transient, $\xi>1$,
moves it outward by $\xi$. Sec.~\ref{sec:rindler} states the same
content as a relation between observables that a single event can
falsify.

\begin{table}[htbp]
\caption{Redshift-stretched timescale
$\tau_{\infty}=\tau_{\mathrm{prop}}/\alpha$ with
$\tau_{\mathrm{prop}}=\xi\,GM/c^{3}$, $\xi=1$, and
$\alpha=\rho/(4M)$, at three proper distances from the horizon. The
bottom two rows give the proper distance inside which a transient is
stretched beyond $N$ cadences,
$\rho_{\max}=4\tau_{\mathrm{prop}}/(N\Delta t_{\mathrm{obs}})$, from
Eq.~\eqref{eq:criterion}; the window scales linearly in $\xi$ and
inversely in $N$.}
\label{tab:rindler}
\begin{ruledtabular}
\begin{tabular}{lcccc}
$\rho\,[GM/c^{2}]$ & $\alpha$ & $\tau_{\infty}$ (M87*) &
$\tau_{\infty}$ (Sgr~A*) & remark \\
\hline
$0.01$ & $0.0025$ & $148$ d & $2.2$ h & deep inside the window \\
$0.10$ & $0.025$  & $14.8$ d & $13$ min & marginal for both \\
$1.00$ & $0.25$   & $1.5$ d & $79$ s & outside the window \\
\hline
\multicolumn{3}{l}{cadence $\Delta t_{\mathrm{obs}}$} & $1$ d & $10$ min \\
\multicolumn{3}{l}{window $\rho$ for $N=10$} & $\lesssim0.15$ & $\lesssim0.013$ \\
\end{tabular}
\end{ruledtabular}
\end{table}

\begin{figure}[htbp]
\includegraphics[width=\linewidth]{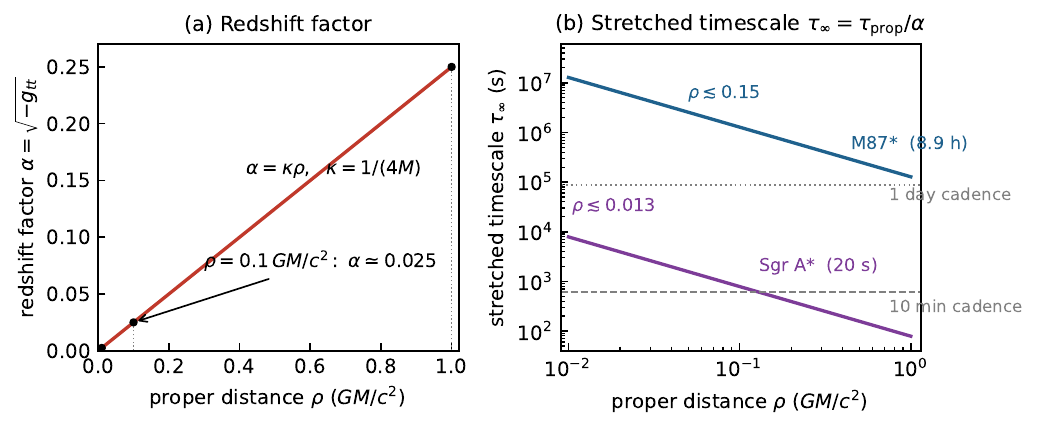}
\caption{Near-horizon kinematics of Eq.~\eqref{eq:rindler}.
(a)~The redshift factor $\alpha=\kappa\rho=\rho/(4M)$ against proper
distance. (b)~The stretched timescale
$\tau_{\infty}=\tau_{\mathrm{prop}}/\alpha$ in seconds for
$\tau_{\mathrm{prop}}=8.9$~h (M87*, $6.5\times10^{9}M_{\odot}$) and
$\tau_{\mathrm{prop}}=20$~s (Sgr~A*, $4\times10^{6}M_{\odot}$), with
the two observing cadences marked. Transients emitted inside
$\rho\lesssim0.15\,GM/c^{2}$ (M87*) and
$\rho\lesssim0.013\,GM/c^{2}$ (Sgr~A*) are stretched beyond ten
cadences, Eq.~\eqref{eq:criterion}. Both curves are in the same
units; the difference between them is the mass.}
\label{fig:rindler}
\end{figure}

\subsection{A falsifiable form of the criterion}

The criterion becomes a test once the observation is written in the
dimensionless variables the source supplies. Let
\begin{equation}
x\equiv\frac{\rho}{GM/c^{2}},\qquad
T\equiv\frac{\tau_{\mathrm{obs}}}{GM/c^{3}},
\label{eq:dimless}
\end{equation}
so that a proper duration $\xi\,GM/c^{3}$ at proper distance $x$
appears with duration
\begin{equation}
x\,T=4\xi ,
\label{eq:hyperbola}
\end{equation}
a hyperbola of unit slope. The two quantities an observation can
supply independently---the duration of a transient and the radius at
which it is localized---are thereby tied together, and either predicts
the other. For Sgr~A*, where $GM/c^{3}=19.7$ s, a ten-minute
transient corresponds to $x=0.13$; for M87*, where
$GM/c^{3}=8.9$ h, a day-long transient corresponds to $x=1.5$ and an
hour-long one to $x=35$.

The statement is falsifiable in one direction and void in the other.
With a tolerance of a factor three in the product---roughly the
accuracy with which both quantities can be had---the model survives an
event whose
\begin{equation}
\frac{4\xi}{3}\ \le\ xT\ \le\ 12\xi ,
\label{eq:band}
\end{equation}
is \emph{falsified} by an event with $xT<4\xi/3$ whose radius and
duration are each measured to better than a factor two, and is simply
\emph{not addressed} by an event with $xT>12\xi$, for which the
duration is set by the source's own variability rather than by the
redshift, so that the window of Eq.~\eqref{eq:criterion} says nothing
about it. The falsifying event is worth stating plainly: a transient
shorter than the redshift permits at the radius where it is seen would
show that durations there are not gravitational, and with that the
reading of the window as ``what the cadence cannot resolve.''

Two things are assumed rather than tested, and they are why this is a
test of the kinematics and not of the paper's main result. The proper
duration must be of order the horizon crossing time, $\xi\sim1$; a
population with $\xi\sim10^{-2}$ would move the hyperbola by two
decades with no observation failing. And the radius must be the radius
of emission, not of the region the emission illuminates; a centroid
displacement measures the former only if the emitting region is
compact. Neither assumption is discharged here. Nor is the MHD content
of Sec.~\ref{sec:groups}, which does not involve the redshift at all
and would survive the test's failure intact.

\subsection{What the illustration does and does not establish}

Three points delimit the claim. First, the numbers in
Table~\ref{tab:rindler} are elementary kinematics of
Eq.~\eqref{eq:rindler}, and they become a prediction only through the
identification of $\tau_{\mathrm{prop}}$ with $\xi\,GM/c^{3}$ together
with the mapping \eqref{eq:hyperbola}; both are assumptions. Second,
whether a flow is ``frozen into''
the horizon material is a separate question about
$\mathrm{Rm}$ and about the magnetic field's topology, and
Proposition~\ref{prop:invariance} is the reason it cannot be answered
by rescaling units: $\mathrm{Rm}$ is invariant under
\eqref{eq:map}, so the freezing must come from the physics of the
flow---a large $\mathrm{Rm}$ at physical values of $\eta$, $\ell$, and
$v$---and not from the transformation. Third, the vacuum
electromagnetic properties of the stretched horizon, such as the
surface resistance $4\pi/c\simeq377\,\Omega$ of the membrane
paradigm, are properties of the membrane model
\cite{damour1978,thorne1986membrane}
and are neither derived nor generalised here.

\section{Discussion}
\label{sec:discussion}

\emph{Relation to similarity solutions.} Self-similar accretion
solutions \cite{narayan1995,narayan2012outflow} and self-similar
magnetic launching of jets \cite{blandfordpayne1982} are scale-invariant
in a stronger sense than the present map: there a discrete or
continuous rescaling carries a \emph{solution} of the balance laws
into another solution, with the radial profiles solving an
eigenvalue problem. Nothing in Sec.~\ref{sec:dictionary} substitutes
for that; Proposition~\ref{prop:covariance} carries equations into
equations, and it is silent about solutions. The distinction is worth
stating explicitly, because the two operations are often conflated:
similarity solutions constrain the exponents of a flow, whereas the
dictionary constrains what a change of units may be said to do.
Scale invariance is also used in a statistical sense, for the exponents
of fluctuation spectra \cite{alberti2022scaling}; that usage
constrains the exponents of a realised cascade, whereas the dictionary
moves no exponent at all, so the two uses of the word do not meet.
Dynamic similarity scaling as used for magnetised implosions
\cite{ruiz2023similarity} belongs with the former class as well: there
the scaling of length, density, and field is chosen so that a specified
combination of the dimensionless groups is held fixed, which is the same
requirement as form invariance, seen from the side of the parameters
rather than from the side of the units.

\emph{Consequences for practice.} The dictionary is useful in three
ways. It is an audit tool: any sentence of the form ``in the limit
$L\to\infty$ the flow becomes ideal / the structure is resolved /
the transient freezes'' can be checked against
Table~\ref{tab:dictionary} and Corollary~\ref{cor:forbidden}
in one line. It is a constraint on numerical zoom procedures, and the constraint
has a size. Refining a grid by $L$ while holding $\eta$, $Q$, and
$\bm j$ fixed is the coordinate-only realisation of
Eq.~\eqref{eq:partial}: it divides $\mathrm{Rm}$ and $S$ by $L$ and
multiplies $t_{A}/t_{\mathrm{diff}}$ by $L$. A hundredfold refinement
therefore lowers $S$ from $10^{4}$ to $10^{2}$ and, since the
Sweet--Parkinson inflow rate scales as $S^{-1/2}$, inflates the
reconnection rate it predicts tenfold. Rescaling $\eta$ but leaving
$\bm j$ at its numerical value is the mirror failure: the Lorentz
force in the momentum equation is then short by a factor $L$, which
is what a partially rescaled restart looks like when it runs.
Implementing the map requires all the rescalings at once, and
implementing it consistently leaves every dimensionless group
unchanged, so nothing about the flow is learned that was not already
in the coarse description. And it fixes the vocabulary: the physically
meaningful statements are the dimensionless ones in
Eq.~\eqref{eq:groups}, and a statement about a dimensionful
coefficient in isolation is not one.

\emph{Limitations.} The map is uniform and acts on all four
coordinates; it is therefore not the near-horizon scaling of a
Killing horizon, which acts on time and on the transverse coordinates
differently. It is a flat-spacetime operation and says nothing about
the extremal near-horizon geometry \cite{bardeenhorowitz1999}, where
the scaling is anisotropic. The derivation assumes Gaussian units,
a linear Ohmic resistivity, and an internal-energy balance closed
by $Q$ and $\Lambda$; a non-Ohmic dissipation channel, or an
anomalous resistivity with its own length dependence, would
modify Eq.~\eqref{eq:heatscale} and require the dictionary to be
re-derived for that closure. Finally the paper is kinematic: it
establishes what the transformation does to equations and to
dimensionless groups, and it does not construct global nonlinear
solutions, which remain the domain of the numerical literature
cited in Sec.~\ref{sec:intro}.

\section{Conclusion}
\label{sec:conclusion}

A uniform dilation of flat spacetime is form-invariant on the closed
MHD balance system, and the price of that form invariance is
exact: the current density, the conductivity, the diffusivity, and
the heating and cooling rates must rescale by fixed factors
(Table~\ref{tab:dictionary}), and no others. The map is an equivalence
transformation rather than a symmetry: what it carries is equations,
together with their coefficients, into equations of the same form,
and the dimensionless groups are left alone only when the
coefficients are carried with them. A coordinate-only zoom, which is
what a grid refinement performs, divides $\mathrm{Rm}$ and $S$ by
$L$ instead, and that defect is measurable rather than microscopic.
The inferences the dilation is often asked to support---that a flow
becomes ideal, that it changes regime, that it is resolved more
finely, that it runs in slow motion---follow from neither branch
(Corollary~\ref{cor:forbidden}), and the mode check of
Sec.~\ref{sec:groups} shows the dictionary reproducing a known
dispersion relation exactly, quality factor included. What remains
are the three readings of Sec.~\ref{sec:readings}, each true in a
restricted form and each with a stated negative.

Genuine scale change in the near-horizon problem comes from the
metric, through the redshift factor $\alpha=\kappa\rho$, and it is
quantitative: transients emitted inside
$\rho\lesssim0.15\,GM/c^{2}$ are stretched beyond a day for M87*,
and inside $\rho\lesssim0.013\,GM/c^{2}$ beyond ten minutes for
Sgr~A*, at the cadences now available
(Tables~\ref{tab:rindler}, Fig.~\ref{fig:rindler}). Which of those
transients a magnetised flow actually produces is a question about
$\mathrm{Rm}$, about field topology, and about the boundary
conditions at the horizon---questions that no rescaling of units
can answer, and that the invariance result of this paper isolates as
the ones genuinely worth asking.

\begin{acknowledgments}
The author thanks the Department of River and Harbor Engineering at
the Nanjing Hydraulic Research Institute for support. The figures were
produced with a Matplotlib script that accompanies this manuscript and
reproduces Table~\ref{tab:rindler} from the closed-form expressions
in Sec.~\ref{sec:rindler}.
\end{acknowledgments}

\section*{Author declarations}
\noindent\textbf{Conflict of interest.} The author has no conflicts
to disclose.

\noindent\textbf{Data availability.} No new data were generated in
this work. The entries of Table~\ref{tab:rindler} and the curves of
Fig.~\ref{fig:rindler} follow from Eqs.~\eqref{eq:tauinf} and
\eqref{eq:criterion} with the masses quoted in the caption, and are
reproduced by the accompanying script.



\begin{thebibliography}{99}
\bibitem{priest2014mhd}
E. Priest, \textit{Magnetohydrodynamics of the Sun} (Cambridge
University Press, Cambridge, 2014).



\bibitem{ji2022reconnection}
H. Ji, W. Daughton, J. Jara-Almonte, A. Le, A. Stanier, and J. Yoo,
``Magnetic reconnection in the era of exascale computing and
multiscale experiments,'' Nat. Rev. Phys. \textbf{4}, 263--282 (2022).
\url{https://doi.org/10.1038/s42254-021-00419-x}



\bibitem{guo2024reconnection}
F. Guo, Y.-H. Liu, S. Zenitani, and M. Hoshino, ``Magnetic reconnection
and associated particle acceleration in high-energy astrophysics,''
Space Sci. Rev. \textbf{220}, 43 (2024).
\url{https://doi.org/10.1007/s11214-024-01073-2}



\bibitem{yao2022vortex}
J. Yao and F. Hussain, ``Vortex reconnection and turbulence cascade,''
Annu. Rev. Fluid Mech. \textbf{54}, 317--347 (2022).
\url{https://doi.org/10.1146/annurev-fluid-030121-125143}



\bibitem{mizuno2022grmhd}
Y. Mizuno, ``GRMHD simulations and modeling for jet formation and
acceleration region in AGNs,'' Universe \textbf{8}, 371 (2022).



\bibitem{jiang2023loops}
H.-X. Jiang, Y. Mizuno, C. M. Fromm, and A. Nathanail,
``Two-temperature GRMHD simulations of black hole accretion flows with
multiple magnetic loops,'' Mon. Not. R. Astron. Soc. \textbf{522},
2399 (2023).



\bibitem{dihingia2022temp}
I. K. Dihingia, Y. Mizuno, C. M. Fromm, and L. Rezzolla,
``Temperature properties in magnetised and radiatively cooled
two-temperature accretion flows onto a black hole,'' Mon. Not. R.
Astron. Soc. \textbf{515}, 1149 (2022).



\bibitem{cattorini2021binary}
F. Cattorini, G. Giacomazzo, F. Haardt, and M. Colpi, ``Fully
general relativistic magnetohydrodynamic simulations
of accretion flows onto spinning massive black hole binary mergers,''
Astrophys. J. \textbf{910}, 78 (2021).



\bibitem{ricarte2023spin}
A. Ricarte, R. Narayan, and B. Curd, ``Recipes for jet feedback and
spin evolution of black holes with strongly magnetised
super-Eddington accretion disks,'' Astrophys. J. Lett.
\textbf{954}, L25 (2023).



\bibitem{aktar2023torus}
A. Aktar, K. C. Pan, and T. Okuda, ``Evolution of MHD torus and mass
outflow around spinning AGNs,'' Mon. Not. R. Astron. Soc.
\textbf{527}, 1745--1759 (2024).



\bibitem{blandfordznajek1977}
R. D. Blandford and R. L. Znajek, ``Electromagnetic extraction of
energy from Kerr black holes,'' Mon. Not. R. Astron. Soc.
\textbf{179}, 433 (1977).



\bibitem{eht2019m87}
The Event Horizon Telescope Collaboration, ``First M87 Event Horizon
Telescope results. I. The shadow of the supermassive black hole,''
Astrophys. J. Lett. \textbf{875}, L1 (2019).



\bibitem{eht2022sgra}
The Event Horizon Telescope Collaboration, ``First Sagittarius A*
Event Horizon Telescope results. I. The shadow of the supermassive
black hole in the center of the Milky Way,'' Astrophys. J. Lett.
\textbf{930}, L12 (2022).



\bibitem{eht2022sgraVI}
The Event Horizon Telescope Collaboration, ``First Sagittarius A*
Event Horizon Telescope results. VI. Testing the black hole metric,''
Astrophys. J. Lett. \textbf{930}, L17 (2022).



\bibitem{eht2024m87}
The Event Horizon Telescope Collaboration, ``The persistent shadow of
the supermassive black hole of M87. I. Observations, calibration,
imaging, and analysis,'' Astron. Astrophys. \textbf{681}, A79 (2024).



\bibitem{broderick2023shadow}
A. E. Broderick, R. Narayan, and P. Tiede, ``Shadow implications:
what does measuring the photon ring imply for gravity?'' Astrophys.
J. \textbf{958}, 114 (2023).



\bibitem{chael2021innershadow}
A. Chael, M. D. Johnson, and A. Lupsasca, ``Observing the inner
shadow of a black hole: a direct view of the event horizon,''
Astrophys. J. \textbf{918}, 6 (2021).
\url{https://doi.org/10.3847/1538-4357/ac09ee}



\bibitem{johnson2023ngeht}
M. D. Johnson \textit{et al.}, ``Key science goals for the
next-generation Event Horizon Telescope,'' Galaxies \textbf{11}, 61
(2023).



\bibitem{pesce2023ngeht}
D. W. Pesce, D. C. M. Palumbo, R. Narayan, M. D. Johnson, N. M.
Nagar, P. Natarajan, and J. L. Gomez, ``Expectations for horizon-scale
supermassive black hole population studies with the next-generation
Event Horizon Telescope,'' Galaxies \textbf{10}, 109 (2022).



\bibitem{narayan1995}
R. Narayan and I. Yi, ``Advection-dominated accretion: self-similarity
and bipolar outflows,'' Astrophys. J. \textbf{452}, 710 (1995).



\bibitem{blandfordpayne1982}
R. D. Blandford and D. G. Payne, ``Hydromagnetic flows from accretion
disks and the production of radio jets,'' Mon. Not. R. Astron. Soc.
\textbf{199}, 883 (1982).



\bibitem{narayan2012outflow}
R. Narayan, A. Sadowski, R. F. Penna, and A. K. Kulkarni, ``GRMHD
simulations of magnetized advection-dominated accretion on a
non-spinning black hole: role of outflows,'' Mon. Not. R. Astron.
Soc. \textbf{426}, 3241 (2012).



\bibitem{ruiz2023similarity}
D. E. Ruiz, P. F. Schmit, M. R. Weis, K. J. Peterson, and M. K.
Matzen, ``Exploring the parameter space of MagLIF implosions using
similarity scaling. III. Rise-time scaling,'' Phys. Plasmas
\textbf{30} (2023). \url{https://doi.org/10.1063/5.0126700}



\bibitem{chrusciel2012review}
P. T. Chru\'sciel, J. P. Costa, and M. Heusler, ``Stationary black
holes: uniqueness and beyond,'' Living Rev. Relativ.
\textbf{15}, 7 (2012).
\url{https://doi.org/10.12942/lrr-2012-7}



\bibitem{bardeenhorowitz1999}
J. M. Bardeen and G. T. Horowitz, ``The extreme Kerr throat
geometry: a vacuum analog of AdS$_2\times$S$^2$,'' Phys. Rev. D
\textbf{60}, 104030 (1999).



\bibitem{galasrichter1991}
F. Galas and E. W. Richter, ``Exact similarity solutions of ideal MHD
equations for plane motions,'' Physica D \textbf{50}, 297--307 (1991).



\bibitem{webb2005lie}
G. M. Webb, G. P. Zank, E. Kh. Kaghashvili, and R. E. Ratkiewicz,
``Magnetohydrodynamic waves in non-uniform flows II: stress-energy
tensors, conservation laws and Lie symmetries,'' J. Plasma Phys.
\textbf{71}, 811--857 (2005).
\url{https://doi.org/10.1017/S0022377805003740}



\bibitem{xia2022lie}
S. Xia and Z. Wang, ``Lie symmetries, group invariant solutions and
conservation laws of ideal MHD equations,'' J. Appl. Anal. Comput.
\textbf{12}, 1959 (2022).



\bibitem{sinuvasan2023}
R. Sinuvasan, A. K. Halder, R. Seshadri, A. Paliathanasis, and
P. G. L. Leach, ``Solutions of magnetohydrodynamics equation through
symmetries,'' Symmetry \textbf{15}, 1908 (2023).
\url{https://doi.org/10.3390/sym15101908}



\bibitem{ovsiannikov1982}
L. V. Ovsiannikov, \textit{Group Analysis of Differential Equations}
(Academic Press, New York, 1982).



\bibitem{damour1978}
T. Damour, ``Black-hole eddy currents,'' Phys. Rev. D
\textbf{18}, 3598--3604 (1978).



\bibitem{thorne1986membrane}
K. S. Thorne, R. H. Price, and D. A. Macdonald, eds.,
\textit{Black Holes: The Membrane Paradigm} (Yale University Press,
New Haven, 1986).



\bibitem{alberti2022scaling}
T. Alberti, S. Benella, V. Carbone, G. Consolini,
V. Quattrociocchi, and M. Stumpo, ``Contrasting scaling properties of
near-Sun sub-Alfv\'enic and super-Alfv\'enic regions,'' Universe
\textbf{8}, 338 (2022).
\url{https://doi.org/10.3390/universe8070338}

\end{thebibliography}
\end{document}